\documentclass[12pt]{article}
\usepackage[british]{babel}
\usepackage{amsmath}
\usepackage{amsthm}
\usepackage{amsfonts}
\usepackage{amssymb}
\usepackage{fancyhdr}
\usepackage[all]{xy}           
\usepackage{graphicx}
\usepackage[left=3cm, right=3cm]{geometry}
\usepackage{enumerate}
\usepackage{cite}
\usepackage{hyperref} 

\theoremstyle{plain}
\newtheorem{thm}{Theorem}[section]
\newtheorem{cor}[thm]{Corollary}
\newtheorem{lem}[thm]{Lemma}
\newtheorem{pro}[thm]{Proposition}

\theoremstyle{definition}

\newtheorem{defi}{Definition}

\theoremstyle{remark}
\newtheorem{remark}[thm]{Remark}

\def\dim{\text{\rm dim}}

\author{Tsiu-Kwen Lee$^\flat$ and Jheng-Huei Lin$^\natural$}

\title{Prime rings having nontrivial centralizers of (skew) traces of Lie ideals}
\date{}

\begin{document}

\maketitle

\centerline {Department of Mathematics, National Taiwan
University${^\flat,}$${^\natural}$}

\centerline {Taipei, Taiwan}

\centerline {tklee@math.ntu.edu.tw$^\flat$; r01221012@ntu.edu.tw${^\natural}$}

\begin{abstract}\vskip6pt
\noindent Let $R$ be a prime ring with center $Z(R)$ and with involution $*$.  Given an additive subgroup $A$ of $R$, let
$
T(A):=\{x+x^*\mid x\in A\}
$
and
$
K_0(A):=\{x-x^*\mid x\in A\}.
$
 Let $L$ be a non-abelian Lie ideal of $R$. It is proved that if $d$ is a nonzero derivation of $R$ satisfying $d(T(L))=0$ (resp. $d(K_0(L))=0$), then $T(R)^2\subseteq Z(R)$ (resp. $K_0(R)^2\subseteq Z(R)$).
These results are applied to the study of $d(T(M))=0$ and $d(K_0(M))=0$ for
  noncentral $*$-subrings $M$ of a division ring $R$ such that $M$ is invariant under all inner automorphisms of $R$, and for
noncentral additive subgroups $M$ of a prime ring $R$ containing a nontrivial idempotent such that $M$ is invariant under all special inner automorphisms of $R$.
The obtained theorems also generalize some recent results on simple artinian rings with involution due to M. Chacron.
 \end{abstract}

{ \hfill\break \noindent 2020 {\it Mathematics Subject Classification.}\ 16N60, 16W10, 16K40. \vskip6pt

\noindent {\it Key words and phrases:}\ Prime ring, division ring, involution, derivation, transpose, symplectic, (skew) trace, Lie ideal, centralizer.\ \vskip6pt


\section{Introduction}
In several branches of mathematics, science and technology, it is always interesting to study how local behavior of a given structure affects the entire one.
On the topic of rings with involution, a lot of researchers have investigated the connection between properties on the whole ring and behavior of a specific subset relating to the involution.
For instance, the study on subsets of the set comprising all symmetric (or skew) elements is very popular.
On such type of subsets, Amitsur et al. considered polynomial identities (see \cite{amitsur1968, baxter1968,herstein1967,martindale1969,montgomery1971a}), Osborn discussed invertibility (see \cite{osborn1967}), and Chacron, Herstein and Montgomery investigated some algebraic conditions such as Jacobson condition of periodicity  (see \cite{chacron1975,herstein1974,herstein1971,montgomery1971b,montgomery1973}).
Inspired by Amitsur \cite{amitsur1968}, similar studies also appear in the theory of group algebras.
See \cite{balogh2012,balogh2008,catino2014,lee2010,lee2015,shalev1992}.
On the other hand, in the theory of functional identities on prime rings, there is a very important theorem relating to this topic, given by Beidar and Martindale, saying that the standard type of functional identities with involution can be completely solved by the standard solution if the maximal degree of all symmetric and skew elements is large enough (see \cite[Theorem 3.1]{beidar1998}).
During the first quarter of the century, there are still several researchers studying different topics on subsets of the set comprising all symmetric (or skew) elements.
See, for example, \cite{bien2019,ferreira2015,goodaire2013,lin2010,mosic2009,siciliano2011,thu2022}.
In 2022, Bien, Hai and Hue investigated how the commutativity of trace or norm elements of units affects the whole ring and the involution (see \cite{bien2022}).
In the recent papers \cite{chacron2022, chacron2023, chacron2024}, Chacron continued the work and discussed the effects of properties, such as commutativity, of trace or skew trace elements of some specific substructures.

Throughout this paper, let $R$ be an associative ring with center $Z(R)$, not necessarily with unity. For $a, b\in R$, let $[a, b]:=ab-ba$, the additive commutator of $a$ and $b$.
For additive subgroups $A, B$ of $R$, we let $AB$ (resp. $[A, B]$) stand for the additive subgroup of $R$ generated by all elements $ab$ (resp. $[a, b]$) for $a\in A$ and $b\in B$. An additive subgroup $L$ is said to be a {\it Lie ideal} of $R$ if $[L, R]\subseteq L$. A Lie ideal is called {\it abelian} if $[L, L]=0$. Otherwise, it is called {\it non-abelian}.
Finally, let $A$ be a subset of $R$, we denote by $\overline A$ the subring of $R$ generated by $A$.

Let $R$ be equipped with an involution $*$, that is, $*$ is an anti-automorphism of $R$ with period $\leq 2$.
An element $x\in R$ is called {\it symmetric} (resp. {\it skew}) if $x^*=x$ (resp. $x^*=-x$).
The involutions on a primitive ring, not a division ring, with nonzero socle are completely determined (see  \cite[Theorem 1]{montgomery1976} and \cite[Theorem 1.2.2]{herstein1976}).
Depending on \cite[Theorem 1.2.2]{herstein1976}, we have the following.\vskip6pt

\begin{defi}\label{def1}
  Let $R$ be a primitive ring, not a division ring, with nonzero socle. An involution $*$ on $R$ is said to be of the {\it transpose} type if $R$ has a symmetric minimal idempotent, that is, there exists a minimal idempotent $e\in R$ such that $e^*=e$. Otherwise, the involution $*$ is said to be of the {\it symplectic} type.
\end{defi}

The following is well-known (see \cite[Theorem 1]{montgomery1976} and \cite[Corollary, p.19]{herstein1976}).

\begin{thm}\label{thm5}
Let $R:=\text{\rm M}_n(D)$ be equipped with an involution $*$, where $n>1$ and $D$ is a division ring.

Case 1:\ The involution $*$ is of the symplectic type. Then $n=2m$ and $D$ is a field.
Let $(A_{ij})\in R=\text{\rm M}_m(\text{\rm M}_2(D))$, where $A_{ij}\in \text{\rm M}_2(D)$, $1\leq i, j\leq m$. We have
$(A_{ij})^*=(B_{ij})$,
where $B_{ij}=A_{ji}^\sigma$ for  $1\leq i, j\leq m$. Here
$$\left[
\begin{array}{cc}
\alpha & \beta \\
\gamma & \delta%
\end{array}%
\right]^\sigma=\left[
\begin{array}{cc}
\delta &- \beta \\
-\gamma & \alpha%
\end{array}%
\right]\in \text{\rm M}_2(D).
$$

Case 2:\ The involution $*$ is of the transpose type. Then there exist an involution $-$ on $D$ and nonzero $\pi_i\in D$, $1\leq i\leq n$, such that
$\bar{\pi_i}=\pi_i$ for all $i$, and
$$
(a_{ij})^*=\text{\rm Diag}(\pi_1^{-1},\ldots,\pi_n^{-1})(\bar{a_{ij}})^t\text{\rm Diag}(\pi_1,\ldots,\pi_n)\in R.
$$
\end{thm}

Let $A$ be an additive subgroup of $R$, and let $S(A)$ (resp. $K(A)$, $T(A)$, $K_0(A)$) be the set of all {\it symmetric} (resp. {\it skew, trace, skew trace}) elements of $A$. Precisely, we have
$$
S(A):=\{x\in A\mid x=x^*\},\ \ K(A):=\{x\in A\mid x^*=-x\},
$$
$$
T(A):=\{x+x^*\mid x\in A\}\ \text{\rm and}\ K_0(A):=\{x-x^*\mid x\in A\}.
$$

A ring $R$ is called {\it prime} if, given $a, b\in R$, $aRb=0$ implies that either $a=0$ or $b=0$. Let $R$ be a prime ring, and let $Q_s(R)$ be the Martindale symmetric ring of quotients of $R$. It is known that $Q_s(R)$ is also a prime ring with center, denoted by $C$, which is called the {\it extended centroid} of $R$. Note that $C$ is always a field. It is also known that if $R$ is a prime PI-ring then $Q_s(R)=RC$ and $IC=RC$ for any nonzero ideal $I$ of $R$.
See \cite{beidar1996} for details.
Let $R$ be a prime ring with involution $*$. It is known that the involution $*$ can be uniquely extended to an involution, denoted by the same $*$, on $Q_s(R)$, and $C^*=C$. We say that $*$ is of the {\it first kind} if $*$ is the identity map on $C$. Otherwise, we say that $*$ is of the {\it second kind}.
\vskip6pt

\begin{defi}\label{def2}
  A prime ring $R$ is called {\it exceptional} if both $\text{\rm char}\,R= 2$ and $\dim_CRC=4$.
\end{defi}


An additive map $d\colon R\to R$ is called a {\it derivation} if $d(xy)=d(x)y+xd(y)$ for all $x, y\in R$. The derivation $d$ of the ring $R$ is called {\it inner} if there exists $a\in R$ such that $d(x)=[a, x]$ for all $x\in R$. It is known that if $R$ is a prime ring, every derivation $d$ of $R$ can be uniquely extended to a derivation, denoted by the same $d$, of $Q_s(R)$. The derivation $d$ of $R$ is called {\it X-inner} if the extension of $d$ to $Q_s(R)$ is inner, that is, there exists $b\in Q_s(R)$ such that $d(x)=[b, x]$ for all $x\in Q_s(R)$. Otherwise, $d$ is called {\it X-outer}.
\vskip6pt

\begin{defi}\label{def3}
  Let $R$ be a prime ring with involution $*$ such that $RC$ is a primitive ring, not a division ring, with nonzero socle.
  We say that $*$ is of the symplectic (resp. transpose) type on $R$ if $*$ is of the symplectic (resp. transpose) type on $RC$.
\end{defi}


Given a unital ring $R$, we let $R^\times$ be the set of all units in the ring $R$. By a {\it $*$-subset} we mean a subset $M$ of $R$ such that $M^*=M$.
In recent papers \cite{chacron2022, chacron2023, chacron2024}, Chacron studied $*$-submonoids $M$ (i.e., multiplicatively closed $*$-subsets with $1$) in a simple artinian ring $R$ such that $M$ is invariant under all inner automorphisms of $R$ (i.e., $uMu^{-1}\subseteq M$ for all $u\in R^\times$).
We mention the main theorems proved in \cite{chacron2023, chacron2024}.

\begin{thm} (\cite[Theorem 6]{chacron2023})\label{thm33}
Let $R$ be a simple artinian ring with involution $*$, $Z(R)\ne \text{\rm GF}(2)$, and let $M$ be a noncentral $*$-submonoid of $R$, which is invariant under all inner automorphisms of $R$. If $[T(M), T(M)]=0$, then $[T(R), T(R)]=0$, $\dim_{Z(R)}R=4$, and $*$ is of the first kind.
\end{thm}

We remark that, in the above theorem, the ring $R$ and the involution $*$ can be described in more details according to either $R\cong \text{\rm M}_2(Z(R))$ or $R$ a $4$-dimensional central division algebra.
The following theorem complements the above theorem under the current $2$-torsion free assumption.

\begin{thm} (\cite[Theorem 6]{chacron2024})\label{thm34}
Let $R$ be a $2$-torsion free simple artinian ring with involution $*$, and let $M$ be a noncentral $*$-submonoid of $R$, which is invariant under all inner automorphisms of $R$. If $[K_0(M), K_0(M)]=0$, then $[K(R), K(R)]=0$, $\dim_{Z(R)}R=4$, and $*$ is of the orthogonal (i.e., transpose) type.
\end{thm}

\begin{remark}\label{remark1} {\rm (i)\
We remark that Corollaries \ref{cor10} and \ref{cor11} generalize both of the above two theorems in the division case and in the non division case, respectively.

(ii)\ In Theorems \ref{thm33} and \ref{thm34}, let $M^+$ denote the additive subgroup of $R$ generated by $M$. Then $M^+$ is a noncentral $*$-subring of $R$, which is invariant under all inner automorphisms of $R$. Moreover, $[T(M), T(M)]=0$ if and only if $[T(M^+), T(M^+)]=0$, and $[K_0(M), K_0(M)]=0$  if and only if $[K_0(M^+), K_0(M^+)]=0$. Hence it suffices to assume from the start that $M$ is a noncentral $*$-subring of $R$, which is invariant under all inner automorphisms of $R$.}
\end{remark}

We first give some observations for the case that $T(R)\subseteq Z(R)$ (resp. $K_0(R)\subseteq Z(R)$).
In a unital ring $R$ with involution $*$, Lim proved that $R$ is a central $*$-trace ring (i.e., $x+x^*\in Z(R)$ for all $x\in R$) if and only if $R$ is a central $*$-norm ring (i.e., $xx^*\in Z(R)$ for all $x\in R$) (see \cite[Theorem 3]{lim1979}). Also, a $*$-commuting ring $R$ (i.e., $[x, x^*]=0$ for all $x\in R$) with $2R=0$ must be a central $*$-trace ring (see \cite[Theorem 2]{lim1979}).
Lim also gave an example to show that in general a $*$-commuting ring is not necessary to be a central $*$-trace ring (see \cite[Example 1, p.127]{lim1979}).
 We remark that the example is an algebra over a field $F$ of characteristic not $2$, which is not a semiprime ring.
Recall that a ring $R$ is called {\it semiprime} if, for $a\in R$, $aRa=0$ implies $a=0$. This is equivalent to saying that $R$ has no nonzero nilpotent ideals.

Given an involution $*$ on a unital semiprime ring $R$, Chacron proved that the following are equivalent:
(1) $R$ is a $*$-commuting ring; (2) $R$ is a central $*$-trace ring;
(3) $R$ is a central $*$-norm ring (see \cite[Theorem 2.3]{chacron2016}).
In fact, commuting anti-automorphisms on semiprime rings, not necessarily with unity, are completely characterized  (see \cite[Theorem 1.2]{lee2017}, and also \cite[Theorem 1.1]{lee2018}).

\begin{pro}\label{pro2}
Let $R$ be a semiprime ring with involution $*$. Then $K_0(R)\subseteq Z(R)$ if and only if $T(R)\subseteq Z(R)$ and $2R\subseteq Z(R)$.
\end{pro}

\begin{proof}
Assume that $K_0(R)\subseteq Z(R)$. Let $x\in R$. Then $x-x^*\in Z(R)$ and so $[x, x^*]=0$. Therefore $R$ is a $*$-commuting ring.
In view of \cite[Theorem 1.1]{lee2018} (or \cite[Theorem 2.3]{chacron2016} for semiprime rings with unity), $R$ is a central $*$-trace ring, i.e., $T(R) \subseteq Z(R)$.
Given $x\in R$, we have
$
2x=(x-x^*)+(x+x^*)\in Z(R)
$
and hence $2R\subseteq Z(R)$. The converse is obvious.
\end{proof}

The following is an immediate consequence of Proposition \ref{pro2}.

\begin{cor}\label{cor8}
Let $R$ be a noncommutative prime ring with involution $*$. Then $K_0(R)\subseteq Z(R)$ if and only if $T(R)\subseteq Z(R)$ and $\text{\rm char}\,R=2$.
\end{cor}

Our first aim is to characterize $[T(R), T(R)]=0$ (resp. $[K_0(R), K_0(R)]=0$) for a given prime ring $R$ with involution $*$.
The following two characterizations are Theorems \ref{thm2} and \ref{thm4}.\vskip6pt

\noindent{\bf Theorem A.}
{\it Let $R$ be a prime ring with involution $*$. If $T(R)\nsubseteq Z(R)$, then the following are equivalent:

(i)\ $[T(R), T(R)]=0$;

(ii)\ $d(T(R))=0$ for some nonzero  derivation $d$ of $R$;

(iii)\ $[b, T(R)]=0$ for some $b\in RC\setminus C$;

(iv)\ $R$ is exceptional and $*$  is of the first kind;

(v)\ $T(R)^2\subseteq Z(R)$.}
\vskip6pt

\noindent{\bf Theorem B.}
{\it Let $R$ be a prime ring with involution $*$. If $K_0(R)\nsubseteq Z(R)$,
then the following are equivalent:

(i)\ $[K_0(R), K_0(R)]=0$;

(ii)\ $d(K_0(R))=0$ for some nonzero  derivation $d$ of $R$;

(iii)\ $[b, K_0(R)]=0$ for some $b\in RC\setminus C$;

(iv)\ $K_0(R)^2\subseteq Z(R)$;

(v)\ $\dim_CRC=4$, $*$ is of the transpose type, and $*$ is of the first kind.}
\vskip6pt

In the statements of Theorems A and B, we need the notion of the transpose or symplectic type on division algebras. It will be
given in Definition \ref{def4} of the next section.
For division rings, we will prove the following main theorem (i.e., Theorem \ref{thm14}).
\vskip6pt

\noindent{\bf Theorem C.}
{\it Let $R$ be a division ring with involution $*$ and let
$M$ be a noncentral $*$-subring of $R$ such that $M$ is invariant under all inner automorphisms of $R$.
If $d$ is a nonzero derivation of $R$ such that
$
d(T(M))=0
$
(resp.\ $d(K_0(M))=0$), then $T(R)^2\subseteq Z(R)$  (resp. $K_0(R)^2\subseteq Z(R)$).}\vskip6pt

Finally, we extend Theorems A and B to the case of Lie ideals (i.e., Theorems \ref{thm20} and \ref{thm32}).\vskip6pt

\noindent{\bf Theorem D.}\ {\it Let $R$ be a prime ring with involution $*$, $L$ a non-abelian Lie ideal of $R$, and $d$ a nonzero derivation of $R$.
If $d(T(L))=0$, then $T(R)^2\subseteq Z(R)$.}\vskip6pt

\noindent{\bf Theorem E.}\ {\it Let $R$ be a prime ring with involution $*$, $L$ a non-abelian Lie ideal of $R$, and $d$ a nonzero derivation of $R$.
If $d(K_0(L))=0$, then $K_0(R)^2\subseteq Z(R)$.}\vskip6pt

By a {\it special inner automorphism} $\phi$ of a ring $R$, not necessarily with unity, we mean that there exists  a square zero element $t\in R$ such that
$
\phi(x)=(1+t)x(1+t)^{-1}
$
for all $x\in R$.
Chuang proved that, in a prime ring $R$ containing a nontrivial idempotent, every noncentral additive subgroup, which is invariant under all special inner automorphisms, contains a Lie ideal of the form $[I, R]$ with $I$ a nonzero ideal of $R$ except when $R$ is exceptional (see \cite[Theorem 1]{chuang1987}).
Note that $[I, R]$ is a non-abelian Lie ideal of $R$ (see Lemma \ref{lem5} (iii)). Applying it, the following is an immediate consequence of Theorems D and E (cf. Theroem \ref{thm19}).\vskip6pt

\begin{cor}\label{cor7}
Let $R$ be a prime ring with involution $*$, containing a nontrivial idempotent.  
Suppose that $A$ is a noncentral additive subgroup of $R$, which is invariant under all special inner automorphisms of $R$.
If $d$ is a nonzero derivation of $R$ such that $d(T(A))=0$ (resp. $d(K_0(A))=0$), then $T(R)^2\subseteq Z(R)$ (resp. $K_0(R)^2\subseteq Z(R)$) except when $R$ is exceptional.
\end{cor}

Whenever it is more convenient, we will use the widely accepted shorthand
\textquotedblleft iff\textquotedblright\ for \textquotedblleft if and only
if\textquotedblright\ in the text.

\section{Preliminaries}
We refer the following lemma to \cite[Lemmas 1 and 5]{lee1995}.

\begin{lem}\label{lem1}
Let $R$ be a noncommutative prime ring with involution $*$.

(i)\ If $d$ is a nonzero derivation of $R$ such that either $d(T(R))\subseteq Z(R)$ or $d(K_0(R))\subseteq Z(R)$, then $\dim_CRC=4$.

(ii)\ If either $T(R)^2\subseteq Z(R)$ or $K_0(R)^2\subseteq Z(R)$, then $\dim_CRC=4$.
\end{lem}

\begin{lem} (\cite[Lemma 5]{bergen1981})\label{lem2}
Let $R$ be a noncommutative prime ring, $\text{\rm char}\,R\ne 2$. If $d$ is a nonzero derivation of $R$,
and $U$ a Lie ideal of $R$ such that
$d(U) = 0$, then $U\subseteq Z(R)$.
\end{lem}

We need a special case of Theorem \ref{thm5} in the proofs below.

\begin{lem}\label{lem7}
Let $R=\text{\rm M}_2(C)$, where $C$ is a field, with involution $*$. \vskip4pt

(i) If the involution $*$ is of the symplectic type, then
$\left[
\begin{array}{cc}
\alpha & \beta \\
\gamma & \delta%
\end{array}%
\right]^*=\left[
\begin{array}{cc}
\delta & -\beta \\
-\gamma & \alpha%
\end{array}%
\right]\in R.$
In this case, we have
$T(R)=C$, $K_0(R)=\{\left[
\begin{array}{cc}
\alpha & 2\beta \\
2\gamma & -\alpha%
\end{array}%
\right]\mid \alpha, \beta, \gamma\in C\}$, and $K_0(R)^2\subseteq C$ iff $\text{\rm char}\,R=2$.
In particular, if $\text{\rm char}\,R\ne 2$, then $K_0(R)=[R, R]$ and $K_0(R)^2=R$.\vskip4pt

(ii)\ If the involution $*$ is of the transpose type, then there exist an involution $-$ on $C$ and $\pi_1, \pi_2\in C\setminus \{0\}$, $\bar{\pi_i}=\pi_i$, $i=1, 2$, such that
$
\left[
\begin{array}{cc}
\alpha & \beta \\
\gamma & \delta%
\end{array}%
\right]^*=\left[
\begin{array}{cc}
\bar\alpha & \pi\bar\gamma \\
 \pi^{-1}\bar\beta & \bar\delta%
\end{array}%
\right],
$
where $\pi:=\pi_1^{-1}\pi_2$. In this case, we have
$$
T(R)=\{\left[
\begin{array}{cc}
\alpha +\bar\alpha & \gamma \\
\pi^{-1}\bar\gamma& \delta+\bar\delta%
\end{array}%
\right]\mid \alpha, \gamma, \delta\in C\}\ \text{\rm and}\ \ K_0(R)=\{\left[
\begin{array}{cc}
\alpha -\bar\alpha & \gamma \\
-\pi^{-1}\bar\gamma& \delta-\bar\delta%
\end{array}%
\right]\mid \alpha, \gamma, \delta\in C\}.
$$

(iii)\ $T(R)\subseteq C$ iff $*$ is of the symplectic type. In this case, $T(R)=C$.
\end{lem}

\begin{lem}\label{lem3}
Let $R=\text{\rm M}_2(C)$, where $C$ is a field, with involution $*$, which is of the transpose type.
If $*$  is of the first kind, then the following hold (with notation in Lemma \ref{lem7}):

(i)\ $
T(R)=\{\left[
\begin{array}{cc}
2\alpha & \gamma \\
\pi^{-1}\gamma& 2\delta%
\end{array}%
\right]\mid \alpha, \gamma, \delta\in C\}$ and $K_0(R)=\{\left[
\begin{array}{cc}
0 & \gamma \\
-\pi^{-1}\gamma & 0%
\end{array}%
\right]\mid  \gamma\in C\};
$

(ii)\ $T(R)^2\subseteq C$ iff $\text{\rm char}\,R=2$. In this case, $T(R)^2=C$;

(iii)\ $K_0(R)^2=C$;

(iv)\ $1\in T(R)$ iff $\text{\rm char}\,R\ne 2$.
\end{lem}

\begin{proof}
(i) follows directly from Lemma \ref{lem7} (ii).

For (ii), if $\text{\rm char}\,R=2$, then
$T(R)=\{\left[
\begin{array}{cc}
0 & \gamma \\
\pi^{-1}\gamma& 0%
\end{array}%
\right]\mid \gamma\in C\}$ and so $T(R)^2=C$.
Conversely, assume that $T(R)^2\subseteq C$. Then
$$
\left[
\begin{array}{cc}
2 &1\\
\pi^{-1}& 2%
\end{array}%
\right]\left[
\begin{array}{cc}
0 & 1 \\
\pi^{-1}& 0%
\end{array}%
\right]=\left[
\begin{array}{cc}
\pi^{-1}&2\\
2\pi^{-1}& \pi^{-1}%
\end{array}%
\right]\in C,
$$
implying $2=0$, that is, $\text{\rm char}\,R=2$. Finally, (iii) and (iv) hold trivially.
\end{proof}

Let $R$ be a $4$-dimensional central division algebra with involution $*$, which is of the first kind, and  let $F$ be a maximal subfield of $R$.
Then, clearly, we have the canonical involution, denoted by the same $*$, on $R\otimes_{Z(R)}F$ as follows:
$
\big(\sum_ix_i\otimes \beta_i\big)^*=\sum_ix_i^*\otimes \beta_i,
$
where $x_i\in R$ and $\beta_i\in F$ for all $i$.

\begin{lem}\label{lem10}
Let $R$ be a $4$-dimensional central division algebra with involution $*$, which is of the first kind, and  let $F$ be a maximal subfield of $R$.
If the canonical involution $*$ on $R\otimes_{Z(R)}F$ is of the symplectic type, then so is the canonical involution $*$ on $R\otimes_{Z(R)}L$ for any maximal subfield $L$ of $R$.
\end{lem}

\begin{proof}
Note that $R\otimes_{Z(R)}F\cong \text{\rm M}_2(F)$, and that
$T(R\otimes_{Z(R)}F)=T(R)\otimes_{Z(R)}F$. Therefore, we have $T(R\otimes_{Z(R)}F)\subseteq F$ iff $T(R)\subseteq Z(R)$.
By Lemma \ref{lem7}, the canonical involution of  $R\otimes_{Z(R)}F$ is of the symplectic type iff $T(R)\subseteq Z(R)$.
Since the condition that $T(R)\subseteq Z(R)$ is independent of the choice of any maximal subfield of $R$, this completes the proof.
\end{proof}

By Lemma \ref{lem10} and its proof, we are ready to give the following. \vskip6pt

\begin{defi}\label{def4}
  Let $R$ be a $4$-dimensional central division algebra with involution $*$, which is of the first kind. 
  The involution $*$ on $R$ is said to be of the symplectic (resp. transpose) type if $T(R)\subseteq Z(R)$ (resp. $T(R)\nsubseteq Z(R)$).
\end{defi}


In Definition \ref{def4}, if $T(R)\subseteq Z(R)$, then the involution $*$ is of the first kind. This fact holds for any noncommutative prime ring $R$ with involution $*$. Indeed, suppose that $\beta\ne \beta^*$ for some $\beta\in C$. There exists a nonzero ideal $I$ of $R$ such that $\beta I\subseteq R$.
For $x, y\in I$, we have
$
\big[\beta x+\beta^*x^*, y\big]=0
$
and
$
\big[\beta^*x+\beta^*x^*, y\big]=0,
$
implying that
$
(\beta^*-\beta)[x, y]=0.
$
Thus $[x, y]=0$ for all $x, y\in I$, implying that $R$ is commutative, a contradiction. This proves that $*$ is of the first kind.

We can also define the symplectic or transpose type of the first kind involution $*$ on a finite-dimensional central division algebra but Definition \ref{def4} is sufficient for our purpose.\vskip6pt

\begin{lem}\label{lem6}
Let $R$ be a $4$-dimensional central division algebra with involution $*$, which is of the first kind.

\noindent Case 1:\ $*$ is of the transpose type. Then the following hold:

(i)\ $T(R)^2\subseteq Z(R)$ iff $\text{\rm char}\,R=2$;

(ii)\ $K_0(R)^2=Z(R)$.

\noindent Case 2:\ $*$ is of the symplectic type. Then $T(R)=Z(R)$, and $K_0(R)^2\subseteq Z(R)$ iff $\text{\rm char}\,R=2$.
\end{lem}

\begin{proof}
Let $\widetilde R:=R\otimes_{Z(R)}L$, where $L$ is a maximal subfield of $R$, $\widetilde T:=T(\widetilde R)$, and ${\widetilde K}_0:=K_0(\widetilde R)$.
Then $*$ is of the first kind on $\widetilde R$. Clearly, $\widetilde T=T(R)\otimes L$ and $\widetilde  {K}_0=K_0(R)\otimes L$. Moreover, we have
$
{\widetilde T}^2=T(R)^2\otimes L
$
and
$
{\widetilde K_0}^2={K_0}(R)^2\otimes L.
$

Case 1:\ $*$ is of the transpose type. Since $*$ is of the transpose type on $\widetilde R$, it follows from Lemma \ref{lem3} (ii) that
$
{\widetilde T}^2\subseteq L\ \text{\rm iff}\ \ \text{\rm char}\,R=2.
$
Moreover, by  Lemma \ref{lem3} (iii) we have ${\widetilde K_0}^2=L$. That is, $T(R)^2\subseteq Z(R)$ iff $\text{\rm char}\,R=2$. Thus (i) is proved. We also have $K_0(R)^2=Z(R)$, proving (ii).

Case 2:\ $*$ is of the symplectic type. It follows directly from Lemma \ref{lem7} (i).
\end{proof}

Let $R$ be a prime ring, and let $a, b\in Q_s(R)$. If $aRb=0$, then either $a=0$ or $b=0$.
Indeed, there exists a nonzero ideal $I$ of $R$ such that $aI, Ib\subseteq R$. Then $aIRIb=0$. The primeness of $R$ implies that either $aI=0$ or $Ib=0$. That is, either $a=0$ or $b=0$.

\begin{lem}\label{lem4}
Let $R$ be a noncommutative prime ring with involution $*$, and let $I$ be a nonzero ideal of $R$. If $aT(I)=0$ (resp. $aK_0(I)=0$) where $a\in Q_s(R)$, then $a=0$.
\end{lem}

\begin{proof}
Assume that $aT(I)=0$, where $a\in Q_s(R)$. Given $x\in R$ and $y\in I$, we have
$$
axy=-a(xy)^*=-ay^*x^*=ayx^*=-axy^*.
$$
That is, $aRT(I)=0$ and so either $a=0$ or $T(I)=0$. Assume that $T(I)=0$.
Then, for $x, y\in I$, we have $xy=-y^*x^*=-yx$ and so $xy+yx=0$. It is easy to prove that $R$ is commutative, a contradiction. Hence $a=0$ follows.
The case $aK_0(I)=0$ has a similar argument by a slight modification.
\end{proof}

\begin{lem}\label{lem8}
Let $R$ be a ring with involution $*$. Then $T(R)^2$ and $K_0(R)^2$ are Lie ideals of $R$.
\end{lem}

\begin{proof}
We only prove that $T(R)^2$ is a Lie ideal of $R$. The skew case has almost the same argument.
Let $t_1, t_2\in T(R)$ and $x\in R$. Then
$$
[t_1t_2, x]=t_1(t_2x+x^*t_2)-(t_1x^*+xt_1)t_2\in T^2.
$$
This proves that $T(R)^2$ is a Lie ideal of $R$.
\end{proof}

\begin{lem}\label{lem5}
Let $R$ be a ring.

(i)\ If $A$ is an additive subgroup of $R$, then $[A, R]=[\overline A, R]$.

(ii)\ If $L$ is a Lie ideal of $R$, then $R[L, L]R\subseteq L+L^2$ and $\big[R[L, L]R, R\big]\subseteq L$ (see \cite[Lemma 2.1]{lee2022}).

(iii)\ If $R$ is a prime ring and if $\big[a, [R, R]\big]=0$ where $a\in Q_s(R)$, then $a\in C$ (see \cite[Lemma 1.5]{herstein1969}).

(iv)\ Let $R$ be a prime ring, and $U$ and $V$ two Lie ideals of $R$. If $[U, V]\subseteq Z(R)$, then either $U\subseteq Z(R)$ or $V\subseteq Z(R)$
except when $R$ is exceptional (\cite[Lemma 7]{lanski1972}).
\end{lem}

We only give the short proof of (i).  Let $a_1,\ldots,a_n\in A$ and $x\in R$. By induction on $n$, we get
\begin{equation*}
[a_1a_2\cdots a_n, x]=[a_2\cdots a_n, xa_1]+ [a_1, a_2\cdots a_nx]\in [A, R].
\end{equation*}
\vskip6pt

\noindent {\bf The Skolem-Noether Theorem.}\ {\it Let $R$ be a finite-dimensional central simple algebra. Then every derivation $d$ of $R$ satisfying $d(Z(R))=0$ is inner.}\vskip6pt

\noindent {\bf The Brauer-Cartan-Hua Theorem.}\ {\it Let $R$ be a division ring, and let $A$ be a subdivision ring of $R$. If $A$ is invariant under all inner automorphisms, then either $A\subseteq Z(R)$ or $A=R$.}\vskip6pt

\section{Theorem A}
In this section, except Lemma \ref{lem9}, we always assume that
{\it $R$ is a prime ring with involution $*$}.
The aim of this section is to characterize $[T(R), T(R)]=0$.

\begin{lem}\label{lem26}
Let $d$ be a nonzero derivation of $R$. If
$d(\beta)=0$ for all $\beta=\beta
^*\in C$, then $d(C)=0$.
\end{lem}

\begin{proof}
Let $\mu\in C$. Then $d(\mu+\mu^*)=0$ and $d(\mu\mu^*)=0$, implying that $(\mu-\mu^*)d(\mu)=0$.
Thus either $\mu=\mu^*$ or $d(\mu)=0$. In either case, we get $d(\mu)=0$. That is, $d(C)=0$.
\end{proof}

\begin{lem}\label{lem12}
Let $R$ be noncommutative, and let $d$ be a nonzero derivation of $R$.
Suppose that either $d(T(R))=0$ or $d(K_0(R))=0$. Then $*$ is of the first kind.
\end{lem}

\begin{proof}
Assume that $d(T(R))=0$. Let $\beta=\beta^*\in C$. There exists a nonzero $*$-ideal $I$ of $R$ such that $\beta I\subseteq R$.
Let $x\in I$. Then $d(x+x^*)=0$ and
$$
0=d( \beta x+(\beta x)^*)=d(\beta (x+x^*))=d(\beta)(x+x^*).
$$
By Lemma \ref{lem4}, we get $d(\beta)=0$. Hence $d(\beta)=0$ for all $\beta=\beta^*\in C$. In view of Lemma \ref{lem26},
we get $d(C)=0$.

Let $\beta\in C$. Then $\beta I\subseteq R$ for some nonzero $*$-ideal $I$ of $R$. Let $x\in I$. Then
$d(x+x^*)=0$ and
$$
0=d( \beta x+(\beta x)^*)= \beta d(x)+\beta^*d(x^*)=(\beta-\beta^*)d(x).
$$
So $(\beta-\beta^*)d(I)=0$. Hence $(\beta-\beta^*)d(IR)=0$ and so $(\beta-\beta^*)Id(R)=0$. Since $d\ne 0$, we get $\beta=\beta^*$.
This proves that  $*$ is of the first kind.

The case $d(K_0(R))=0$ has almost the same argument by a slight modification.
\end{proof}

The following is Theorem A mentioned in the introduction.

\begin{thm}\label{thm2}
 If $T(R)\nsubseteq Z(R)$, then the following are equivalent:

(i)\ $[T(R), T(R)]=0$;

(ii)\ $d(T(R))=0$ for some nonzero  derivation $d$ of $R$;

(iii)\ $[b, T(R)]=0$ for some $b\in RC\setminus C$;

(iv)\ $R$ is exceptional and $*$  is of the first kind;

(v)\ $T(R)^2\subseteq Z(R)$.
\end{thm}

\begin{proof}
In view of Lemma \ref{lem1}, if one of (i)--(v) is satisfied, then  $\dim_CRC=4$.
Also, by Lemma \ref{lem12}, any one of (i)--(iv) implies that $*$ is of the first kind.

Let $F:=C$ if $RC\cong \text{\rm M}_2(C)$, and let $F$ be a maximal subfield of $RC$ if $RC$ is a division algebra. Let
$\widetilde R:=RC\otimes_CF\cong\text{\rm M}_2(F)$. Since $T(R)\nsubseteq Z(R)$,
in either case, any one of  (i)--(iv) implies that
the canonical involution $* $ on $\widetilde R$ is of the transpose type (see Lemma \ref{lem7}). That is, the involution $*$ on $RC$ is of the transpose type.

(i) $\Rightarrow$ (ii):\ Since $T(R)\nsubseteq Z(R)$, we have $[b, T(R)]=0$ for some $b\in T(R)\setminus Z(R)$.
Let $d$ be the inner derivation of $R$ induced by the element $b$. Then $d$ is nonzero and $d(T(R))=0$, as desired.

(ii) $\Rightarrow$ (iii):\ Recall that $\dim_CRC=4$.
Let $\beta=\beta^*\in C$. There exists a nonzero $*$-ideal $I$ of $R$ such that $\beta I\subseteq R$.
Then, for $x\in I$, we have
$$
0=d((\beta x)+(\beta x)^*)=d(\beta(x+x^*))=d(\beta)(x+x^*).
$$
That is, $d(\beta)(x+x^*)=0$ for all $x\in I$. It follows from Lemma \ref{lem4} that $d(\beta)=0$.
Therefore, $d(\beta)=0$ for all $\beta=\beta^*\in C$.
In view of Lemma \ref{lem26},
we get $d(C)=0$.

By the Skolem-Noether theorem, the extension of $d$ to $RC$ is inner. That is, there exists $b\in RC\setminus C$ such that $d(x)=[b, x]$ for
all $x\in R$. By (ii), it follows that $[b, T(R)]=0$.

(iii) $\Rightarrow$ (iv):\
In view of Lemma \ref{lem8}, $T(R)^2$ is a Lie ideal of $R$. Since $[b, T(R)]=0$, we have $[b, T(R)^2]=0$.
If $\text{\rm char}\,R\ne 2$, by Lemma \ref{lem2} we see that $T(R)^2\subseteq Z(R)$.
Hence, by Lemmas \ref{lem3} and  \ref{lem6}, we get $\text{\rm char}\,R=2$, a contradiction.
Therefore, $\text{\rm char}\,R=2$ and so $R$ is an exceptional prime ring.

(iv) $\Rightarrow$ (v):\ In view of Lemmas \ref{lem3} and \ref{lem6}, it follows that
$T(R)^2\subseteq Z(R)$.

(v) $\Rightarrow$ (i):\
Let $t_1, t_2\in T(R)\setminus \{0\}$.
Since $T(R)^2\subseteq Z(R)$, we have
$
0=[t_1t_2, t_2]=[t_1, t_2]t_2,
$
and so $[t_1, t_2]t_2T(R)=0$. Note that $0\ne t_2T(R)\subseteq Z(R)$ (see Lemma \ref{lem4}). Thus $[t_1, t_2]=0$.
Therefore, $[T(R), T(R)]=0$.
\end{proof}

\begin{cor}\label{cor5}
Let $d$ be a nonzero derivation of $R$. If $d(T(R))=0$, then $T(R)^2\subseteq Z(R)$.
\end{cor}

\begin{proof}
If $T(R)\subseteq Z(R)$, it is clear that $T(R)^2\subseteq Z(R)$.
Suppose that $T(R)\nsubseteq Z(R)$. Since $d(T(R))=0$, it follows from Theorem \ref{thm2} that $T(R)^2\subseteq Z(R)$, as desired.
\end{proof}

Given a subset $A$ of a ring $R$,
let $\mathfrak{C}_R(A)$ denote the {\it centralizer} of $A$ in $R$, that is,
$$
\mathfrak{C}_R(A)=\{x\in R\mid xa=ax\ \forall a\in A\}.
$$

 A prime ring $R$ is called {\it centrally closed} if $R=RC$.
The following corollary itself is interesting.

\begin{cor}\label{cor1}
If $R$ is centrally closed, then $\mathfrak{C}_R(T(R))$ is equal to $R$, $Z(R)$, or $T(R)+Z(R)$.
\end{cor}

\begin{proof}
Clearly, if $T(R)\subseteq Z(R)$, then $\mathfrak{C}_R(T(R))=R$. Assume that $T(R)\nsubseteq Z(R)$.
If $\mathfrak{C}_R(T(R))\subseteq Z(R)$, then $\mathfrak{C}_R(T(R))=Z(R)$. Hence we assume that $\mathfrak{C}_R(T(R))\nsubseteq Z(R)$.
That is, there exists $b\in R\setminus Z(R)$ such that $[b, T(R)]=0$. In view Theorem \ref{thm2} (iv), $R$ is exceptional and $*$  is of the first kind.
In this case, $Z(R)=C$. Also, one of the following holds:

Case 1:\ $R\cong \text{\rm M}_2(C)$. Since $T(R)\nsubseteq Z(R)$, $*$ is of the transpose type on $R$. As given in Lemma \ref{lem3},
$$
T(R)=\{\left[
\begin{array}{cc}
0 & \gamma \\
\pi^{-1}\gamma& 0%
\end{array}%
\right]\mid \gamma\in C\}.
$$
Let $x\in R$. Then $x\in \mathfrak{C}_R(T(R))$ iff $\big[x, \left[
\begin{array}{cc}
0 & 1\\
\pi^{-1}&0%
\end{array}%
\right]\big]=0$. A direct computation shows that $x\in T(R)+C$. That is, $ \mathfrak{C}_R(T(R))=T(R)+Z(R)$.

Case 2:\ $R$ is a $4$-dimensional central division algebra, and $\text{\rm char}\,R=2$. Since $T(R)\nsubseteq Z(R)$, $*$ is of the transpose type on $R$.
Let $F$ be a maximal subfield of $R$, and let $\widetilde R:=R\otimes_{Z(R)}F\cong \text{\rm M}_2(F)$.  As given in Case 1, we have
$$
 \mathfrak{C}_{\widetilde R}(T({\widetilde R}))=T({\widetilde R})+Z({\widetilde R})=T(R)\otimes F+Z(R)\otimes F.
$$
On the other hand,
$
 \mathfrak{C}_{\widetilde R}(T({\widetilde R}))= \mathfrak{C}_{\widetilde R}(T(R)\otimes F)=\mathfrak{C}_R(T(R))\otimes F.
$
Hence we get $\mathfrak{C}_R(T(R))=T(R)+Z(R)$.
\end{proof}

\begin{lem}\label{lem9}
Let $R$ be a prime ring with a Lie ideal $L$. Then $\mathfrak{C}_R(L)=R$, $\mathfrak{C}_R(L)=Z(R)$, or both $R$ is exceptional and
$LC=[a, RC]=Ca+C$ for any $a\in L\setminus Z(R)$.
\end{lem}

\begin{proof}
Let $x\in \mathfrak{C}_R(L)$ and $r\in R$. Then $[x, L]=0$, $[r, L]\subseteq L$ and so
$$
\big[[x, r], L\big]\subseteq \big[[x, L], r\big]+\big[x, [r, L]\big]\subseteq \big[x, L\big]=0.
$$
That is, $[\mathfrak{C}_R(L), R]\subseteq \mathfrak{C}_R(L)$, proving that $\mathfrak{C}_R(L)$ is a Lie ideal of $R$.
Clearly, $[\mathfrak{C}_R(L), L]=0$.
In view of Lemma \ref{lem5} (iv) and \cite[Lemma 6.1]{lee2025}, $\mathfrak{C}_R(L)\subseteq Z(R)$, $L\subseteq Z(R)$, or both $R$ is exceptional and
$LC=[a, RC]=Ca+C$ for any $a\in L\setminus Z(R)$. The first two cases imply that $\mathfrak{C}_R(L)=Z(R)$ and $\mathfrak{C}_R(L)=R$, respectively.
The proof is now complete.
\end{proof}

\begin{cor}\label{cor3}
If $R$ is centrally closed, then $\mathfrak{C}_R(T(R)^2)$ is equal to either $R$ or $Z(R)$.
\end{cor}

\begin{proof}
In view of Lemma \ref{lem8}, $T(R)^2$ is a Lie ideal of $R$. By Lemma \ref{lem9}, we have $\mathfrak{C}_R(T(R)^2)=R$, $\mathfrak{C}_R(T(R)^2)
=Z(R)$, or both $R$ is exceptional and $T(R)^2C=[a, RC]=Ca+C$
for any $a\in T(R)^2\setminus C$. Assume the third case. Then $\dim_CT(R)^2C=2$ and $T(R)\nsubseteq Z(R)$. In view Theorem \ref{thm2}, $*$ is of the second kind. There exists $\beta\in C$ such that $\beta\ne \beta^*$. Let $I$ be a nonzero $*$-ideal of $R$ such that $\beta I\subseteq R$.
Then, for $x\in I$,
$$
(\beta^*-\beta)T(R)x=T(R)\big(\beta^*(x+x^*)-(\beta x+(\beta x)^*)\big)\subseteq T(R)^2C.
$$
Applying the same argument, we get $(\beta^*-\beta)^2yx\subseteq T(R)^2C$ for all $x, y\in I$. Hence $I^2C  \subseteq T(R)^2C$.
Note that $I^2C=R$. Therefore $T(R)^2C=R$, a contradiction. This completes the proof.
\end{proof}

\section{Theorem B}
In this section we always assume that
{\it $R$ is a prime ring with involution $*$}.
Relative to Theorem \ref{thm2}, we have another corresponding skew trace result, i.e., Theorem B mentioned in the introduction.

\begin{thm}\label{thm4}
If $K_0(R)\nsubseteq Z(R)$,
then the following are equivalent:

(i)\ $[K_0(R), K_0(R)]=0$;

(ii)\ $d(K_0(R))=0$ for some nonzero  derivation $d$ of $R$;

(iii)\ $[b, K_0(R)]=0$ for some $b\in RC\setminus C$;

(iv)\ $K_0(R)^2\subseteq Z(R)$;

(v)\ $\dim_CRC=4$, $*$ is of the transpose type, and $*$ is of the first kind.
\end{thm}

\begin{proof}
We first consider the case that $\text{\rm char}\,R=2$. In this case, we have $T(R)=K_0(R)$. It follows from Theorem \ref{thm2} that (i)--(iv) are equivalent. Moreover, by Lemma \ref{lem12}, $*$ is of the first kind.
Since $T(R)=K_0(R)\nsubseteq Z(R)$, it follows that $*$ is of the transpose type. So (i) implies (v). By Theorem \ref{thm2}, (v) also implies (i)
and hence we are done in this case.

Hence we always assume that $\text{\rm char}\,R\ne 2$. In view of Lemma \ref{lem1}, if one of (i)--(iv) is satisfied, then
$\dim_CRC=4$.

(i) $\Rightarrow$ (ii):\ Since $K_0(R)\nsubseteq Z(R)$, choose $b\in K_0(R)\setminus Z(R)$. Let $d(x)=[b, x]$ for $x\in R$.
Then $d$ is a nonzero derivation of $R$ and $d(K_0)=0$.

(ii) $\Rightarrow$ (iii):\ We have a similar argument as given in the proof of ``(ii) $\Rightarrow$ (iii)'' of Theorem \ref{thm2}.

(iii) $\Rightarrow$ (iv):\ Note that $K_0(R)^2$ is a Lie ideal of $R$ (see Lemma \ref{lem8}). Since $[b, K_0(R)]=0$, we get $[b, K_0(R)^2]=0$.
In view of Lemma \ref{lem2}, it follows that $K_0(R)^2\subseteq Z(R)$, as desired.

(iv) $\Rightarrow$ (i):\ Let $k_1, k_2\in K_0(R)\setminus \{0\}$. Then $[k_1k_2, k_2]=0$ and so
$[k_1, k_2]k_2K_0(R)=0$. In view of Lemma \ref{lem4}, $0\ne k_2K_0(R)\subseteq Z(R)$. We conclude that $[k_1, k_2]=0$, as desired.

Up to now, we have proved that (i)--(iv) are equivalent. In this case, it follows from Lemma \ref{lem12} that $*$ is of the first kind.

(iv) $\Rightarrow$ (v):\ Suppose on the contrary that $*$ is of the symplectic type on $RC$.
Hence, by Lemma \ref{lem7} (i) and Case 2 of Lemma \ref{lem6}, $K_0(R)^2\subseteq Z(R)$ iff $\text{\rm char}\,R=2$, a contradiction. Thus $*$ is of the transpose type.

(v) $\Rightarrow$ (iv):\ By Lemma \ref{lem3} and Case 1(ii) of Lemma \ref{lem6}, we get $K_0(R)^2\subseteq C$.
\end{proof}

\begin{cor}\label{cor6}
Let $d$ be a nonzero derivation of $R$. If $d(K_0(R))=0$, then $K_0(R)^2\subseteq Z(R)$.
\end{cor}

\begin{cor}\label{cor2}
If $R$ is centrally closed, then $\mathfrak{C}_R(K_0(R))$ is equal to $R$, $Z(R)$, or $K_0(R)+Z(R)$.
\end{cor}

\begin{proof}
Clearly, if $K_0(R)\subseteq Z(R)$, then $\mathfrak{C}_R(K_0(R))=R$. Assume that $K_0(R)\nsubseteq Z(R)$.
If $\mathfrak{C}_R(K_0(R))\subseteq Z(R)$, then $\mathfrak{C}_R(K_0(R))=Z(R)$. Hence we assume that $\mathfrak{C}_R(K_0(R))\nsubseteq Z(R)$.
That is, there exists $b\in R\setminus Z(R)$ such that $[b, K_0(R)]=0$.
In view Theorem \ref{thm4} (v), $\dim_CR=4$, $*$ is of the transpose type on $R$, and $*$ is of the first kind.
In this case, $Z(R)=C$. Also, one of the following holds:

Case 1:\ $R\cong \text{\rm M}_2(C)$. As given in Lemma \ref{lem3},
$$
K_0(R)=\{\left[
\begin{array}{cc}
0 & \gamma \\
-\pi^{-1}\gamma & 0%
\end{array}%
\right]\mid  \gamma\in C\}.
$$
Let $x\in R$. Then $x\in \mathfrak{C}_R(K_0(R))$ iff $\big[x, \left[
\begin{array}{cc}
0 & 1\\
-\pi^{-1}&0%
\end{array}%
\right]\big]=0$. A direct computation shows that $x\in K_0(R)+C$. That is, $ \mathfrak{C}_R(K_0(R))=K_0(R)+Z(R)$.

Case 2:\ $R$ is a $4$-dimensional central division algebra. Note that $*$ is of the transpose type on $R$.
Let $F$ be a maximal subfield of $R$, and let $\widetilde R:=R\otimes_{Z(R)}F\cong \text{\rm M}_2(F)$.  As given in Case 1, we have
$$
 \mathfrak{C}_{\widetilde R}(K_0({\widetilde R}))=K_0({\widetilde R})+Z({\widetilde R})=K_0(R)\otimes F+Z(R)\otimes F.
$$
On the other hand,
$
 \mathfrak{C}_{\widetilde R}(K_0({\widetilde R}))= \mathfrak{C}_{\widetilde R}(K_0(R)\otimes F)=\mathfrak{C}_R(K_0(R))\otimes F.
$
Hence we get $\mathfrak{C}_R(K_0(R))=K_0(R)+Z(R)$.
\end{proof}

\begin{cor}\label{cor4}
If $R$ is centrally closed, then $\mathfrak{C}_R(K_0(R)^2)$ is equal to either $R$ or $Z(R)$.
\end{cor}

\begin{proof}
In view of Lemma \ref{lem8}, $K_0(R)^2$ is a Lie ideal of $R$. By Lemma \ref{lem9}, we have $\mathfrak{C}_R(K_0(R)^2)=R$, $\mathfrak{C}_R(K_0(R)^2)
=Z(R)$, or both $R$ is exceptional and $K_0(R)^2C=[a, RC]=Ca+C$
for any $a\in K_0(R)^2\setminus C$. Assume the third case, implying that $\dim_CK_0(R)^2C=2$ and $K_0(R)\nsubseteq Z(R)$.

Suppose on the contrary that $*$ is of the second kind. There exists $\beta\in C$ such that $\beta\ne \beta^*$. Let $I$ be a nonzero $*$-ideal of $R$ such that $\beta I\subseteq R$.
Then, for $x\in I$,
$$
(\beta^*-\beta)K_0(R)x=K_0(R)\big(\beta^*(x-x^*)-(\beta x-(\beta x)^*)\big)\subseteq K_0(R)^2C.
$$
Applying the same argument, we get $(\beta^*-\beta)^2yx\subseteq K_0(R)^2C$ for all $x, y\in I$. Hence $I^2C  \subseteq K_0(R)^2C$.
Note that $I^2C=R$. Therefore $K_0(R)^2C=R$, a contradiction. This proves that $*$ is of the first kind.

If $*$ is of the transpose type on $R$, it follows from Theorem \ref{thm4} that $K_0(R)^2\subseteq C$, a contradiction. Thus $*$ is of the symplectic type.
By Lemmas \ref{lem7} and \ref{lem6}, it follows that $\text{\rm char}\,R\ne 2$, a contradiction. This completes the proof.
\end{proof}

\section{Theorem C}
Let $R$ be a prime ring with involution $*$. We say that $R$ satisfies  a
$*$-generalized polynomial $f$ if $f(X_1,\ldots,X_n, Y_1,\ldots,Y_n)$ is a generalized polynomial in noncommutative variables
$X_i$ and $Y_i$, $1\leq i\leq n$, with coefficients in $Q_s(R)$ and
$
f(x_1,\ldots,x_n, x_1^*,\ldots,x_n^*)=0
$
for all $x_i\in R$ (see \cite{chuang1989}). We also say that $R$ satisfies the
$*$-generalized polynomial $f(X_1,\ldots,X_n, X_1^*,\ldots,X_n^*)$. If $f$ has coefficients in $C$, we say that $R$ satisfies the $*$-polynomial
$f(X_1,\ldots,X_n, X_1^*,\ldots,X_n^*)$.

\begin{pro} ( \cite[Proposition 4]{chuang1989}) \label{pro1}
Suppose that $R$ is a prime ring with involution $*$, and that $f$ is a $*$-generalized polynomial. If $f$ vanishes on a nonzero ideal of $R$,
then $f$ vanishes on $Q_s(R)$.
\end{pro}

To keep the statements of our theorems below neat, throughout this section, we always make the following assumption:\vskip4pt

{\it Let $R$ be a division ring with involution $*$ and let
$M$ be a noncentral $*$-subring of $R$ such that $M$ is invariant under all inner automorphisms of $R$.}\vskip4pt

We begin with the following.

\begin{thm}\label{thm9}
If $M$ is a PI-ring, then the following hold:

(i)\ $R=MZ(R)$, where $Z(R)$ is the quotient field of $Z(M)$;

(ii)\ If $R$ is equipped with involution $*$, then $R$ and $M$ satisfy the same
$*$-generalized polynomials with coefficients in $R$.
\end{thm}

\begin{proof}
(i)\ We claim that $[M, M]\ne 0$. Otherwise, $[M, M]=0$.
Let $m\in M$ and $u\in R\setminus \{0, -1\}$. Then there exist $m_1, m_2\in M$ such that
$$
um=m_1u\ \ \text{\rm and}\ \ (1+u)m=m_2(1+u).
$$
Then $m=(m_2-m_1)u+m_2$ and so $(m_2-m_1)[m, u]=0$. Therefore, either $m_1=m_2$ or $[m, u]=0$.
The former case implies that $m=m_2$ and so $[m, u]=0$. This implies that $m\in Z(R)$. Thus $M\subseteq Z(R)$, a contradiction.

By \cite[Theorem 2 and Corollary 1]{rowen1973}, $Z(M)\ne 0$, $MC_M$ is a finite-dimensional central simple $C_M$-algebra and $Z(MC_M)=C_M$, where $C_M$ is the quotient field of $Z(M)$. In this case, $MC_M$ is a subdivision ring of $R$. Moreover, $uC_Mu^{-1}\subseteq C_M$
 for all $0\ne u\in R$, implying that $MC_M$ is invariant under all inner automorphisms of $R$.
Since $[M, M]\ne 0$, $MC_M$ is a noncentral subdivision ring of $R$, it follows from
 the Brauer-Cartan-Hua theorem that $R=MC_M$.
Clearly, $C_M=Z(R)$ and so $Z(R)$ is the quotient field of $Z(M)$.

(ii)\ It follows directly from (i) and Proposition \ref{pro1}.
\end{proof}

\begin{thm}\label{thm10}
If
$
[T(M), T(M)]=0\ \ (resp.\ [K_0(M), K_0(M)]=0),
$
then $T(R)^2\subseteq Z(R)$  (resp. $K_0(R)^2\subseteq Z(R)$).
\end{thm}

\begin{proof}
Assume that
\begin{eqnarray}
[x+x^*, y+y^*]=0,
\label{eq:3}
\end{eqnarray}
 for all $x, y\in M$. Since the prime ring $M$
satisfies the $*$-polynomial $[X+X^*, Y+Y^*]$, in view of \cite[Theorem 1]{amitsur1969}, $M$ is a PI-ring.
 It follows from Theorem \ref{thm9} (ii) that Eq.\eqref{eq:3} holds for all $x\in R$. That is, $[T(R), T(R)]=0$. In view of Theorem \ref{thm2}, either $T(R)\subseteq Z(R)$ or $T(R)^2\subseteq Z(R)$.  In either case, $T(R)^2\subseteq Z(R)$.
A similar argument can be applied to proving that $[K_0(R), K_0(R)]=0$ implies $K_0(R)^2\subseteq Z(R)$ by Theorem \ref{thm4}. Hence the proof is complete.
\end{proof}

The following generalizes the division case of Theorems \ref{thm33} and Theorems \ref{thm34}.

\begin{cor}\label{cor10}
(i)\ If
$
[T(M), T(M)]=0,
$
then $\dim_{Z(R)}R=4$ and $*$ is of the first kind.

(ii)\ If
$
[K_0(M), K_0(M)]=0
$
and $\text{\rm char}\,R\ne 2$,
then $\dim_{Z(R)}R=4$, $*$ is of the transpose type, and $*$ is of the first kind.
\end{cor}

\begin{proof}
(i)\ In view of Theorem \ref{thm10}, we have $T(R)^2\subseteq Z(R)$ and hence $\dim_{Z(R)}R=4$ (see Lemma \ref{lem1} (ii)).
If $T(R)\subseteq Z(R)$, then $*$ is of the first kind (see the argument below Definition \ref{def4}). Suppose next that $T(R)\nsubseteq Z(R)$.
It follows from Theorem \ref{thm2} (iv) that $*$ is of the first kind.

(ii)\  In view of Theorem \ref{thm10}, we have $K_0(R)^2\subseteq Z(R)$ and hence $\dim_{Z(R)}R=4$ (see Lemma \ref{lem1} (ii)).
If $K_0(R)\subseteq Z(R)$, then, by Corollary \ref{cor8}, $T(R)\subseteq Z(R)$ and $\text{\rm char}\,R=2$, a contradiction.
Thus $K_0(R)\nsubseteq Z(R)$. It follows from Theorem \ref{thm4} (v) that $*$ is of the transpose type and $*$ is of the first kind.
\end{proof}

The following is the inner case of Theorem C.

\begin{thm}\label{thm13}
If
$
\mathfrak{C}_R(T(M))\nsubseteq Z(R)
$
(resp.\ $\mathfrak{C}_R(K_0(M))\nsubseteq Z(R))$,
 then $T(R)^2\subseteq Z(R)$  (resp. $K_0(R)^2\subseteq Z(R)$).
\end{thm}

\begin{proof}
We only prove the trace case.
The skew trace case is similar by a slight modification.
Since $\mathfrak{C}_R(T(M))\nsubseteq Z(R)$, there exists $b\in R\setminus Z(R)$ such that $[b, x+x^*]=0$ for all
$x\in M$.
Then $[b, T(M)^2]=0$. Note that $T(M)^2$ is a Lie ideal of $M$ (see Lemma \ref{lem8}).

We claim that $[T(M)^2, T(M)^2]=0$. Suppose not. By Lemma \ref{lem5} (ii), we have
$$
I:=M[T(M)^2, T(M)^2]M\subseteq T(M)^2+T(M)^4
$$
is a nonzero ideal of $M$. Then $[b, I]=0$. Since $[b, IM]=0$, we get $I[b, M]=0$ and so $[b, M]=0$.

Note that $\mathfrak{C}_R(M)$ is a subdivision ring of $R$ and $u\mathfrak{C}_R(M)u^{-1}\subseteq \mathfrak{C}_R(M)$ for all $u\in R^\times$.
By the Brauer-Cartan-Hua theorem, either $\mathfrak{C}_R(M)\subseteq Z(R)$ or $\mathfrak{C}_R(M)=R$.
The latter case implies that $M$ is central, a contradiction. Thus $\mathfrak{C}_R(M)\subseteq Z(R)$ and so
 $b\in Z(R)$, a contradiction. This proves our claim.

Therefore, $M$ satisfies a nontrivial $*$-polynomial and so $M$ is a PI-ring (see \cite[Theorem 1]{amitsur1969}).
In view of Theorem \ref{thm9} (ii), we get $[b, T(R)]=0$ since $[b, x+x^*]=0$ for all
$x\in M$. We are now done by Theorem \ref{thm2}.
\end{proof}

The following is Theorem C mentioned in the introduction.

\begin{thm}\label{thm14}
If $d$ is a nonzero derivation of $R$ such that
$
d(T(M))=0
$
(resp.\ $d(K_0(M))=0$), then $T(R)^2\subseteq Z(R)$  (resp. $K_0(R)^2\subseteq Z(R)$).
\end{thm}

\begin{proof}
Assume that $d(T(M))=0$. Then $d(T(M)^4)=0$. Since $T(M)^2$ is a Lie ideal of the prime ring $M$ (see Lemma \ref{lem8}), it follows from \cite[Theorem 1.6]{lee2025} that either
$T(M)^4$ contains a nonzero ideal, say $I$, of $M$ or $[T(M)^2, T(M)^2]=0$.

Case 1:\ $T(M)^4$ contains a nonzero ideal $I$ of $M$.
Then $d(I)=0$ and so $d(IM)=0$. This implies that $Id(M)=0$ and so $d(M)=0$.
Let $m\in M$ and $u\in R^\times$. Then $umu^{-1}\in M$ and so $d(umu^{-1})=0$.
Therefore, $d(u)mu^{-1}+umd(u^{-1})=0$. Note that $d(u^{-1})=-u^{-1}d(u)u^{-1}$. We see that $[u^{-1}d(u), m]=0$. That is,
$[u^{-1}d(u), M]=0$. In particular, $[u^{-1}d(u), T(M)]=0$. If $u^{-1}d(u)\notin Z(R)$ for some $u\in R^\times$, it follows from Theorem \ref{thm13} that
$T(R)^2\subseteq Z(R)$. Suppose next that $u^{-1}d(u)\in Z(R)$ for all $u\in R^\times$. Since $R$ is a division ring, we get $[x, d(x)]=0$ for all $x\in R$.
By Posner's theorem (see \cite[Theorem 2]{posner1957}), either $d=0$ or $R$ is commutative, a contradiction.

Case 2:\  $[T(M)^2, T(M)^2]=0$. That is, $M$ satisfies the $*$-polynomial
$$
f:=\big[(X_1+X_1^*)(X_2+X_2^*), (Y_1+Y_1^*)(Y_2+Y_2^*)\big].
$$
In view of \cite[Theorem 1]{amitsur1969}, $M$ is a PI-ring. In view of Theorem \ref{thm9} (ii), $R$ also satisfies the $*$-polynomial $f$.
Hence $R$ is a division PI-ring, implying $\dim_{Z(R)}R<\infty$.

Let $\beta\in Z(M)$ and $x\in M$. Then $d(x+x^*)=0$, $d(\beta+\beta^*)=0$, and so
\begin{eqnarray}
  0 &=&d(\beta x+(\beta x)^*)\nonumber \nonumber \\
    &=&d(\beta)x+\beta d(x)+\beta^*d(x^*)+d(\beta^*)x \nonumber \\
   &=&(\beta-\beta^*)d(x).\nonumber
\end{eqnarray}
Thus, either $\beta=\beta^*$ or $d(x)=0$. If $d(M)=0$, then, by applying the same argument in Case 1, we are done in this case.
Hence we may assume that $\beta=\beta^*$ for all $\beta\in Z(M)$.

Let $\beta\in Z(M)$ and $x\in M$. Then $d(\beta x+(\beta x)^*)=0$ and so $d(\beta(x+x^*))=0$, implying $d(\beta)(x+x^*)=0$.
That is, $d(\beta)T(M)=0$. Note that $T(M)\ne 0$ since $M$ is noncentral. So $d(\beta)=0$. Thus $d=0$ on $Z(M)$. This implies that $d=0$ on $Z(R)$ as $Z(R)$ is the quotient field of $Z(M)$. By the Skolem-Noether theorem, $d$ is inner on $R$. By Theorem \ref{thm13}, $T(R)^2\subseteq Z(R)$.
This completes the trace case.

The skew trace case has almost the same argument by a slight modification.
\end{proof}

\section{Lie ideals: Theorem D}
We begin with the case of ideals.

\begin{thm}\label{thm17}
Let $R$ be a noncommutative prime ring with involution $*$, $d$ a nonzero derivation of $R$, and $I$ a nonzero ideal of $R$.
If $d(T(I))=0$ (resp. $d(K_0(I))=0$), then $T(R)^2\subseteq Z(R)$ (resp. $K_0(R)^2\subseteq Z(R)$).
\end{thm}

\begin{proof}
Replacing $I$ by $I\cap I^*$, we may assume that $I$ is a nonzero $*$-ideal of $R$.

Case 1: Assume that $d(T(I))=0$.
Let $x\in I$ and $w\in T(I)$. Then
$$
xw\in I, w=w^*,d(x^*)=-d(x), d(w)=0\ \text{\rm and}\ d(xw+(xw)^*)=0.
$$
We have
\begin{eqnarray*}
  0 &=&d(xw+(xw)^*)\nonumber\\
     &=&d(x)w+xd(w)+w^*d(x^*)+d(w^*)x^*\nonumber\\
     &=&d(x)w-wd(x).\nonumber
\end{eqnarray*}
 That is, $[d(I), y+y^*]=0$ for all $y\in I$. Applying  Proposition \ref{pro1}, we get $[d(I), y+y^*]=0$ for all $y\in R$. In view of Theorem \ref{thm2}, either $d(I)\subseteq Z(R)$ or $T(R)^2\subseteq Z(R)$.
 If $d(I)\subseteq Z(R)$, then, by \cite[Theorem 2]{lee1986}, either $d=0$ or $R$ is commutative, a contradiction. Thus $T(R)^2\subseteq Z(R)$.

Case 2: Assume that $d(K_0(I))=0$. By Case 1, we may assume that $\text{\rm char}\,R\ne 2$. Let $x\in I$ and $w\in K_0(I)$. Then
$$
xw\in I, w^*=-w, d(x^*)=d(x), d(w)=0\ \text{\rm and}\ d(xw-(xw)^*)=0.
$$
We have
\begin{eqnarray*}
  0 &=&d(xw-(xw)^*)\nonumber\\
     &=&d(x)w+xd(w)-w^*d(x^*)-d(w^*)x^*\nonumber\\
     &=&d(x)w+wd(x).\nonumber
\end{eqnarray*}
 That is, $[d(I), w_1w_2]=0$ for all $w_1, w_2\in K_0(I)$. It follows from \cite[Theorem 4]{bergen1981} that $w_1w_2\in Z(R)$ for all $w_1, w_2\in K_0(I)$.
 That is, $K_0(I)^2\subseteq Z(R)$.
  Applying  Proposition \ref{pro1}, we get $K_0(R)^2\subseteq Z(R)$, as desired.
\end{proof}

Herstein proved that, in a prime ring $R$ containing a nontrivial idempotent, every noncentral subring invariant under all special inner automorphisms of $R$ contains a nonzero ideal (see \cite[Theorem, p.26]{herstein1983}).
Hence the following is an immediate consequence of Theorem \ref{thm17} (cf. Corollary \ref{cor7}).

\begin{thm} \label{thm19}
Let $R$ be a prime ring with involution $*$, possessing a nontrivial idempotent. 
Suppose that $M$ is a noncentral subring of $R$, which is invariant under all special inner automorphisms of $R$.
If $d$ is a nonzero derivation of $R$ such that $d(T(M))=0$ (resp. $d(K_0(M))=0$), then $T(R)^2\subseteq Z(R)$ (resp. $K_0(R)^2\subseteq Z(R)$) except when $R\cong \text{\rm M}_2(\text{\rm GF}(2))$.
\end{thm}

The following generalizes the non division case of Theorems \ref{thm33} and Theorems \ref{thm34}.

\begin{cor}\label{cor11}
Let $R$ be a prime ring with involution $*$, possessing a
nontrivial idempotent, and $R\ncong \text{\rm M}_2(\text{\rm GF}(2))$. Suppose that $M$ is a noncentral subring of $R$, which is invariant under all special inner automorphisms of $R$.

(i)\ If
$
[T(M), T(M)]=0,
$
then $\dim_CRC=4$ and $*$ is of the first kind.

(ii)\ If
$
[K_0(M), K_0(M)]=0
$
and $\text{\rm char}\,R\ne 2$,
then $\dim_CRC=4$, $*$ is of the transpose type, and $*$ is of the first kind.
\end{cor}

\begin{proof}
In view of \cite[Theorem, p.26]{herstein1983}, $M$ contains a nonzero ideal $I$ of $R$. Replacing $I$ by $I\cap I^*$, we may assume that $I$
is a nonzero $*$-ideal of $R$.

(i)\ Since
$
[T(M), T(M)]=0,
$
it follows that $[T(I), T(I)]=0$. By Proposition \ref{pro1}, we get $[T(R), T(R)]=0$. This implies that $\dim_CRC=4$
(see Lemma \ref{lem1} (ii) and Theorem \ref{thm2} (iv)).
  If $T(R)\nsubseteq Z(R)$, then it follows from Theorem \ref{thm2} (iv) that $*$ is of the first kind.
Suppose that $T(R)\subseteq Z(R)$. By Lemma \ref{lem7}, $*$ is of the symplectic type and hence is of the first kind.

(ii)\ Since
$
[K_0(M), K_0(M)]=0,
$
it follows that $[K_0(I), K_0(I)]=0$. By Proposition \ref{pro1}, we get $[K_0(R), K_0(R)]=0$.
If $K_0(R)\subseteq Z(R)$, then, by Corollary \ref{cor8}, $T(R)\subseteq Z(R)$ and $\text{\rm char}\,R=2$, a contradiction.
Thus $K_0(R)\nsubseteq Z(R)$. It follows from Theorem \ref{thm4} (v) that $\dim_CRC=4$, $*$ is of the transpose type and $*$ is of the first kind.
\end{proof}

We now turn to the case of Lie ideals. The following will play a key role in the proofs below.

\begin{thm} (\cite[Theorem 5.2 Case 1]{lee2025})\label{thm23}
Let $R$ be a prime ring, $L$ a non-abelian Lie ideal of $R$, and $a\in R\setminus Z(R)$.
Then $L+aL$ contains a nonzero ideal of $R$.
\end{thm}

\begin{lem} \label{lem24}
Let $R$ be a prime ring with involution $*$, $\text{\rm char}\,R=2$, and $b\in R\setminus Z(R)$. Suppose that
$
\big[b, [x, y]+[x, y]^*\big]=0
$
for all $x, y\in R$.
Then $T(R)^2\subseteq Z(R)$.
\end{lem}

\begin{proof}
By hypothesis, we have
\begin{eqnarray}
\big[b, [x, y]+[x, y]^*\big]=0 \label{eq:5}
\end{eqnarray}
for all $x, y\in R$.

Suppose first that $b^*=b$. Replacing $x$ by $b$ in Eq.\eqref{eq:5}, we get
$
\big[b, [b, y]+[b, y^*]\big]=0
$
and hence
$[b^2, y+ y^*]=0$ for all $y\in R$. That is, $[b^2, T(R)]=0$. In view of Theorem \ref{thm2}, either $T(R)^2\subseteq Z(R)$ or $b^2\in Z(R)$.
Assume that $b^2\in Z(R)$. Thus $[b, z]b=b[b, z]$ for all $z\in R$.

Note that $\big[[R, R], [R, R]\big]\ne 0$.
In view of Theorem \ref{thm23}, $I\subseteq [R, R]+b[R, R]$ for some nonzero ideal $I$ of $R$.
Let $x\in I$. Write
$$
x=\sum_i[a_i, b_i]+\sum_jb[c_j, d_j]
$$
for some finitely many $a_i, b_i, c_j, d_j\in R$.
Then
\begin{eqnarray*}
[b, x+x^*]&=&\Big[b, \sum_i([a_i, b_i]+[a_i, b_i]^*)\Big]+\Big[b, \sum_jb[c_j, d_j]+\sum_j[c_j, d_j]^*b\Big]\\
                  &=&b\Big[b, \sum_j[c_j, d_j]\Big]+\Big[b, \sum_j[c_j, d_j]^*\Big]b\\
                  &=&b\Big[b, \sum_j[c_j, d_j]\Big]+b\Big[b, \sum_j[c_j, d_j]^*\Big]\\
                  &=&b\sum_j\Big[b,[c_j, d_j]+[c_j, d_j]^*\Big]\\
                  &=&0.
\end{eqnarray*}
That is, $[b, T(I)]=0$ and so, by Theorem \ref{thm17}, we get $T(R)^2\subseteq Z(R)$ in the case of $b^*=b$.

We consider the general case.
By Eq.\eqref{eq:5}, $\big[b+b^*, [x, y]+[x, y]^*\big]=0$ for all $x, y\in R$. By the case above, either $T(R)^2\subseteq Z(R)$ or $b+b^*\in Z(R)$.  Assume that $b^*=b+\beta$. Then $b^*=b+\beta$ for some $\beta\in Z(R)$.
It suffices to prove that $\beta=0$. Suppose on the contrary that $\beta\ne 0$.

Replacing $x$ by $b$ in Eq.\eqref{eq:5}, we get $\big[b, [b, y]+[b^*, y^*]\big]=0$ and hence
$\big[b, [b, y]+[b+\beta, y^*]\big]=0$, implying
$[b^2, y+ y^*]=0$ for all $y\in R$. That is, $[b^2, T(R)]=0$. In view of Theorem \ref{thm2}, we have $b^2\in Z(R)$.
Replacing $x$ by $bx$  in Eq.\eqref{eq:5}, we get
\begin{eqnarray}
\big[b, [bx, y]+[bx, y]^*\big]=0. \label{eq:6}
\end{eqnarray}
Expanding Eq.\eqref{eq:6}, we have
\begin{eqnarray*}
0&=&\big[b, [b, y]x+b[x, y]+[x^*(b+\beta), y^*]\big]\\
  &=&[b, y][b, x]+b\big[b, [x, y]\big]+\big[b, x^*[b, y^*]\big]+\big[b, [x^*, y^*]\big](b+\beta)\\
  &=&[b, y][b, x]+b\big[b, [x, y]\big]+[b, x^*][b, y^*]+(b+\beta)\big[b, [x^*, y^*]\big]\\
   &=&[b, y][b, x]+[b, x^*][b, y^*]+\beta\big[b, [x^*, y^*]\big]+b\Big(\big[b, [x, y]+[x^*, y^*]\big]\Big)\\
   &=&[b, y][b, x]+[b, x^*][b, y^*]+\beta\big[b, [x^*, y^*]\big].\\
\end{eqnarray*}
That is,
\begin{eqnarray}
[b, y][b, x]+[b, x^*][b, y^*]+\beta\big[b, [x^*, y^*]\big]=0 \label{eq:7}
\end{eqnarray}
for all $x, y\in R$. Replacing $y$ by $by$ in Eq.\eqref{eq:7}, we get
\begin{eqnarray}
[b, by][b, x]+[b, x^*][b, y^*(b+\beta)]+\beta\big[b, [x^*, y^*(b+\beta)]\big]=0.
 \label{eq:8}
\end{eqnarray}
Note that $[b, x^*][b, y^*(b+\beta)]=(b+\beta)[b, x^*][b, y^*]$ and
$$
\big[b, [x^*, y^*(b+\beta)]\big]=(b+\beta)\big[b, [x^*, y^*]\big]+[b, y^*][x^*, b].
$$
Thus it follows from Eq.\eqref{eq:8} that
\begin{eqnarray*}
b[b, y][b, x]+(b+\beta)[b, x^*][b, y^*]+\beta(b+\beta)\big[b, [x^*, y^*]\big]+\beta[b, y^*][x^*, b]=0
\end{eqnarray*}
and so
\begin{eqnarray*}
&&b\Big([b, y][b, x]+[b, x^*][b, y^*]
+\beta\big[b, [x^*, y^*]\big]\Big)\\
&&\ \ \ \ \ \ \ \ \ \ \ \ \ \ +\beta\Big([b, x^*][b, y^*]+[b, y^*][x^*, b]\Big)+\beta^2\big[b, [x^*, y^*]\big]=0.
\end{eqnarray*}
By Eq.\eqref{eq:7}, we get
$[b, x^*][b, y^*]+[b, y^*][x^*, b]+\beta\big[b, [x^*, y^*]\big]=0$ and hence
\begin{eqnarray}
[b, x][b, y]+[b, y][b, x]+\beta\big[b, [x, y]\big]=0
 \label{eq:9}
\end{eqnarray}
for all $x, y\in R$. Replacing $x$ by $bx$ in Eq.\eqref{eq:9}, we get
$$
b\Big([b, x][b, y]+[b, y][b, x]+\beta\big[b, [x, y]\big]\Big)+\beta [b, y][b,x]=0
$$
and so $[b, y][b,x]=0$ for all $x, y\in R$. The primeness of $R$ implies that $b\in Z(R)$, a contradiction.
\end{proof}

\begin{lem}\label{lem14}
Let $R$ be a noncommutative prime ring with involution $*$, $\text{\rm char}\,R=2$, and $d$ a nonzero derivation of $R$.
If $d(T(R))\subseteq Z(R)$, then $*$ is of the first kind.
\end{lem}

\begin{proof}
Suppose that $*$ is of the second kind. Let $\beta\in C$ be such that $\beta\ne \beta^*$. Choose a nonzero $*$-ideal $I$ of $R$ such that
$\beta^2I\subseteq R$. Then $d(\beta^2)=0$ and $d((\beta^*)^2)=0$.
Let $x\in I$. We compute
\begin{eqnarray*}
&&d(\beta^2x+(\beta^2 x)^*)\\
&=& d(\beta^2 (x+x^*)+((\beta^2)^*+\beta^2)x^*)\\
&=& \beta^2 d(x+x^*)+d(\beta^2)(x+x^*)+d((\beta^*)^2+\beta^2)x^*+((\beta^*)^2+\beta^2)d(x^*)\\
&=& \beta^2 d(x+x^*)+((\beta^*)^2+\beta^2)d(x^*)\in Z(R)
\end{eqnarray*}
and so $((\beta^*)^2+\beta^2)d(x^*)\in Z(R)$. Since $(\beta^*)^2+\beta^2\ne 0$, we get $d(x^*)\in Z(R)$ for all $x\in I$.
As $I=I^*$, we have $d(I)\subseteq Z(R)$, implying that $R$ is commutative (see \cite[Theorem 2]{lee1986}), a contradiction. This completes the proof.
\end{proof}

\begin{lem}\label{lem25}
Let $R$ be a ring with involution $*$, and let $L$ be a Lie ideal of $R$.
Then $[K_0(L), K_0(R)] \subseteq K_0(L)$.
\end{lem}

\begin{proof}
Let $u \in L$ and $r \in R$. Then
  $$
  [u-u^{\ast}, r-r^{\ast}] = ([u,r]-[u,r]^{\ast})-([u,r^{\ast}]-[u,r^{\ast}]^{\ast}) \in K_0(L).
  $$
Thus $[K_0(L), K_0(R)] \subseteq K_0(L)$, as desired.
\end{proof}

The following is Theorem D.

\begin{thm}\label{thm20}
Let $R$ be a prime ring with involution $*$, $L$ a non-abelian Lie ideal of $R$, and $d$ a nonzero derivation of $R$.
If $d(T(L))=0$, then $T(R)^2\subseteq Z(R)$.
\end{thm}

\begin{proof}
Case 1:\ $\text{\rm char}\,R\ne 2$.
Let $u\in L$ and $t\in T(L)$. Then $[u, t]\in L$, $[u, t]^*=[t, u^*]$ and $d(t)=0$. Thus
\begin{eqnarray*}
0&=&d([u, t]+[u, t]^*)\\
  &=&d([u, t]+[t, u^*])\\
  &=&[d(u), t]+[t, d(u^*)]\\
  &=&[d(u), t]-[t, d(u)]\\
  &=&2[d(u), t].
\end{eqnarray*}
That is,
$
\big[T(L), d(L)\big]=0.
$
In view of \cite[Theorem 2]{lee1983}, $T(L)\subseteq Z(R)$.
Let $u\in L$ and $r\in R$. Then $u+u^*\in Z(R)$ and
$$
[u, r]+[u, r]^*=[u, r]+[r^*, u^*]=[u, r]-[r^*, u]=[u, r+r^*]\in Z(R),
$$
implying $[L, T(R)]=0$.
Since $L$ is noncentral, it follows from Corollary \ref{cor5} that $T(R)^2\subseteq Z(R)$.

Case 2:\ $\text{\rm char}\,R=2$.
Let $J:=R[L, L]R$ and $I:=J\cap J^*$. Then $I$ is a nonzero $*$-ideal of $R$ such that $[I, R]\subseteq L$ (see Lemma \ref{lem5} (ii)). Since $R$ is noncommutative, $[I, R]$ is a non-abelian Lie ideal of $R$.
By the fact that $d(T(L))=0$, we get $d(T([I, R]))=0$.

Suppose first that $d$ is X-inner. There exists $b\in Q_s(R)$ such that $d(x)=[b, x]$ for all $x\in R$. Hence
\begin{eqnarray}
\big[b, [x, r]+[x, r]^*\big]=0
\label{eq:11}
\end{eqnarray}
for all $x\in I$ and $r\in R$. In view of Proposition \ref{pro1}, Eq.\eqref{eq:11} holds for all $x, r\in Q_s(R)$. By Lemma \ref{lem24}, $T(Q_s(R))^2\subseteq C$ and so
$T(R)^2\subseteq Z(R)$, as desired.

We next consider the general case. In view of Lemma \ref{lem25},
$[T(L),T(R)] \subseteq T(L)$, which implies $[T(L), d(T(R))]=0$. Applying the X-inner case above, either $T(R)^2\subseteq Z(R)$ or $d(T(R))\subseteq Z(R)$. For the latter case,
  assume that  $d(T(R))\subseteq Z(R)$.
It follows from Lemma \ref{lem1} that $\dim_CRC=4$ as $R$ is noncommutative.
In view of Lemma \ref{lem14}, $*$ is of the first kind.

Up to now, we have proved that $R$ is exceptional and $*$ is of the first kind.
It follows from Theorem \ref{thm2} that $T(R)^2\subseteq Z(R)$.
\end{proof}

\section{Lie ideals: Theorem E}
Let $R$ be a prime ring with involution $*$, and $d$ a derivation of $R$. We let $d^*$ be the derivation of $R$ defined by
$
d^*(x)=d(x^*)^*
$
for $x\in R$. We say that $d$ is a {\it $*$-derivation} if $d=d^*$ and is a {\it skew $*$-derivation} if $d=-d^*$. Note that if $b\in S(R)$ then the inner derivation defined by $b$ is a skew $*$-derivation. Similarly, if $b\in K(R)$, then the inner derivation defined by $b$ is a $*$-derivation.

\begin{lem} \label{lem28}
Let $R$ be a prime ring with involution $*$, $\text{\rm char}\,R\ne 2$, and $b\in R\setminus Z(R)$. Suppose that
$
\big[b, [x, y]-[x, y]^*\big]=0
$
for all $x, y\in R$.
Then $K_0(R)^2\subseteq Z(R)$.
\end{lem}

\begin{proof}
By hypothesis,
\begin{eqnarray}
\big[b, [x, y]-[x, y]^*\big]=0
\label{eq:12}
\end{eqnarray}
for all $x, y\in R$.
We claim that $*$ is of the first kind. Otherwise, $\beta\ne \beta^*$ for some $\beta\in C$. Choose a nonzero $*$-ideal $I$ of $R$ such that $\beta \subseteq R$. Let $x, y\in I$. Then $\beta^*\big[b, [x, y]-[x, y]^*\big]=0$ and
$$
\big[b, [\beta x, y]-[\beta x, y]^*\big]=0,
$$
implying that $(\beta^*-\beta)[b, [x, y]]=0$. That is, $\big[b, [I, I]\big]=0$ and so $b\in Z(R)$, a contradiction.

Suppose first that $b^*=-b$.  Then the inner derivation of $R$ induced by $b$ is a $*$-derivation.
Replacing $x$ by $b$ in Eq.\eqref{eq:12},
$$
0=\big[b, [b, y]-[b, y]^*\big]=\big[b, [b, y]-[y^*, b^*]\big]=\big[b, [b, y-y^*]\big]
$$
for all $y\in R$. In view of \cite[Theorem 2.6]{lin1986}, it follows that $\dim_CRC=4$.

Suppose that $*$ is of the symplectic type. Then $T(R)\subseteq Z(R)$. Let $x, y\in R$. Then $[x, y]+[x, y]^*\in Z(R)$. Together with Eq.\eqref{eq:12}, we get
$
\big[b, 2[x, y]\big]=0.
$
Since $\text{\rm char}\,R\ne 2$, $[b, [R, R]]=0$ and so $b\in Z(R)$, a contradiction. Thus $*$ is of the transpose type.
In view of Theorem \ref{thm4}, $K_0(R)^2\subseteq Z(R)$, as desired.

For the general case, it follows from Eq.\eqref{eq:12} that
$$
\big[b^*-b, [x, y]-[x, y]^*\big]=0
$$
for all $x, y\in R$. Suppose on the contrary that $K_0(R)^2\nsubseteq Z(R)$. By the skew case proved above, $b^*-b\in Z(R)$. Since $*$ is of the first kind, we have $b=b^*$.

Replacing $x$ by $b$ in Eq.\eqref{eq:12}, we get
$$
0=\big[b, [b, y]-[b, y]^*\big]=\big[b, [b, y+y^*]\big]
$$
for all $y\in R$. In particular, $\big[b, [b, [x, y]+[x, y]^*]\big]=0$ for all $x, y\in R$. Together with Eq.\eqref{eq:12}, we get
$[b, [b, [x, y]]=0$ for all $x, y\in R$. In view of \cite[Theorem 1]{bergen1981}, it follows that $b\in Z(R)$, a contradiction.
Hence $K_0(R)^2\subseteq Z(R)$.
\end{proof}

Finally, we prove the following theorem, i.e., Theorem E.\vskip6pt

\begin{thm}\label{thm32}
Let $R$ be a prime ring with involution $*$, $L$ a non-abelian Lie ideal of $R$, and $d$ a nonzero derivation of $R$.
If $d(K_0(L))=0$, then $K_0(R)^2\subseteq Z(R)$.
\end{thm}

\begin{proof}
Case 1:\ $\text{\rm char}\,R=2$. Then $K_0(L)=T(L)$. It follows from Theorem \ref{thm20} that $K_0(R)^2=T(R)^2\subseteq Z(R)$, as desired.

Case 2:\ $\text{\rm char}\,R\ne 2$. In view of Lemma \ref{lem25},  $[K_0(L), K_0(R)] \subseteq K_0(L)$.
Since $d(K_0(L))=0$, we get
$[K_0(L), d(K_0(R))]=0$. As $L$ is a non-abelian Lie ideal of $R$, there exists a nonzero $*$-ideal $I$ of $R$ such that
$[I, R]\subseteq L$ (see Lemma \ref{lem5} (ii)). Let $z\in d(K_0(R))$. Then
\begin{eqnarray}
\big[z, [x, y]-[x, y]^*\big]=0
\label{eq:13}
\end{eqnarray}
for all $x\in I$ and $y\in R$. In view of Proposition \ref{pro1}, Eq.\eqref{eq:13} holds for all $x, y\in R$.
By Lemma \ref{lem28}, either $K_0(R)^2\subseteq Z(R)$ or $z\in Z(R)$ for all $z\in d(K_0(R))$. Assume the latter case.
That is, $d(K_0(R))\subseteq Z(R)$.
In view of Lemma \ref{lem1} (i), it follows that $\dim_CRC=4$.

Step 1:\ $d(C)=0$. Let $\beta=\beta^*\in C$. Choose a nonzero $*$-ideal $J$ of $R$ such that $\beta J\subseteq I$ and $J\subseteq I$.
Let $x\in I$ and $r\in J$. Then $d([x, r]-[x, r]^*)=0$ and
$$
0=d([x, \beta r]-[x, \beta r]^*)=d(\beta([x, r]-[x, r]^*))=d(\beta)([x, r]-[x, r]^*).
$$
Suppose on the contrary that $d(\beta)\ne 0$. Then $[x, r]=[x, r]^*$ for all $x\in I$ and $r\in J$.
Since $xr\in J$, we get
$$
[x, xr]=[x, xr]^*=(x[x, r])^*=[x, r]^*x^*=[x, r]x^*.
$$
Hence
$[x, [x, J]]=0$ for all $x=x^*\in I$. In view of \cite[Theorem 1]{bergen1981}, $x\in Z(R)$ for all $x=x^*\in I$.
It is easy to prove that $R$ is commutative, a contradiction.

Up to now, we have proved that if $\beta=\beta^*\in C$, then $d(\beta)=0$.
In view of Lemma \ref{lem26},
we get $d(C)=0$.
It follows from the Skolem-Noether theorem that $d$ is X-inner, that is, there exists $b\in Q_s(R)$ such that $d(x)=[b, x]$ for all $x\in R$.

Step 2:\ The involution $*$ is of the first kind. Let $\beta\in C$. The aim is to prove that $\beta^*=\beta$. Choose a nonzero ideal $J$ of $R$ such that $\beta J\subseteq R$.
Then, for $x\in I$ and $r\in J$, we have
$
\big[b, [x, r]-[x, r]^*\big]=0
$
and
$$
\big[b, [x, \beta r]-[x, \beta r]^*\big]=\beta^*\big[b, [x, r]-[x, r]^*\big]-(\beta^*-\beta)\big[b, [x, r]\big]=0.
$$
Therefore $(\beta^*-\beta)\big[b, [x, r]\big]=0$. That is, $(\beta^*-\beta)\big[b, [I, J]\big]=0$.
By Lemma \ref{lem2} and $b\notin Z(R)$, it follows that $\beta^*=\beta$, as desired.

Up to now, we have proved that $\dim_CRC=4$ and $*$ is of the first kind.

Suppose that $*$ is of the symplectic type, that is, $T(R)\subseteq Z(R)$.
Let $x\in I$ and $r\in R$. Then
$\big[b, [x, r]-[x, r]^*\big]=0$ and $[x, r]+[x, r]^*\in Z(R)$, implying that $\big[b, 2[x, r]\big]=0$.
That is, $[b, [I, R]]=0$ and so $b\in Z(R)$, a contradiction.

This proves that $*$ is of the transpose type. It follows from Theorem \ref{thm4} that $K_0(R)^2\subseteq Z(R)$, as desired.
\end{proof}


\begin{thebibliography}{99}
\bibitem{amitsur1968} S. A. Amitsur, {\it Rings with involution}, Israel J. Math. {\bf 6} (1968), 99--106.

\bibitem{amitsur1969}  S. A. Amitsur, {\it Identities in rings with involutions}, Israel J. Math. {\bf 7}(1) (1969), 63--68.

\bibitem{balogh2012} Z. Balogh, {\it Lie derived length and involutions in group algebras}, J. Pure Appl. Algebra {\bf 216} (2012), 1282--1287.

\bibitem{balogh2008} Z. Balogh and T. Juh\'{a}sz, {\it Derived lengths of symmetric and skew symmetric elements in group algebras}, JP J. Algebra Number T. {\bf 12}(2) (2008), 191--203.

\bibitem{baxter1968} W. E. Baxter and W. S. Martindale III, {\it Rings with involution and polynomial identities}, Canad. J. Math. {\bf 20} (1968), 465--473.

\bibitem{beidar1998} K. I. Beidar and W. S. Martindale III, {\it On functional identities in prime rings with involution}, J. Algebra {\bf 203} (1998), 491--532.

\bibitem{beidar1996} K. I. Beidar, W. S. Martindale III and A. V. Mikhalev,
``Rings with Generalized Identities", Monographs and Textbooks in
Pure and Applied Mathematics, {\bf 196}. Marcel Dekker, Inc., New
York, 1996.

\bibitem{bergen1981} J. Bergen, I. N. Herstein and J. W. Kerr, {\it Lie ideals and derivations of prime rings}, J. Algebra {\bf 71}(1) (1981), 259--267.

\bibitem{bien2019} M. H. Bien, {\it A note on local commutators in division rings with involution}, Bull. Korean Math. Soc. {\bf 56}(3) (2019), 659--666.

\bibitem{bien2022} M. H. Bien, B. X. Hai and D. T. Hue, {\it On the unit groups of rings with involution}, Acta Math. Hungar. {\bf 166}(2) (2022), 432--452.

\bibitem{catino2014} F. Catino, G. T. Lee and E. Spinelli, {\it Group algebras whose symmetric elements are Lie metabelian}, Forum Math. {\bf 26}(5) (2014), 1459--1471.

\bibitem{chacron2016}  M. Chacron, {\it Commuting involution}, Comm. Algebra {\bf 44}(9) (2016),  3951--3965.

\bibitem{chacron2022} M. Chacron, {\it Certain invariant multiplicative subset of a simple Artinian ring with involution}, JP J. Algebra Number T. {\bf 53}(2) (2022), 151--163.

\bibitem{chacron2023} M. Chacron, {\it Certain invariant multiplicative subset of a simple Artinian ring with involution. II}, Acta Math. Hungar. {\bf 170}(2) (2023), 437--454.

\bibitem{chacron2024} M. Chacron, {\it On invariant submonoids}, Comm. Algebra, published online in 2024,
https://doi.org/10.1080/00927872.2024.2400198.

\bibitem{chacron1975} M. Chacron and I. N. Herstein, {\it Powers of skew and symmetric elements in division rings}, Houston J. Math. {\bf 1}(1) (1975), 15--27.

\bibitem{chuang1987}  C.-L. Chuang, {\it On invariant additive subgroups},
Israel J. Math.   {\bf 57}(1) (1987), 116--128.

\bibitem{chuang1989}  C.-L. Chuang, {\it $*$-differential identities of prime rings with involution}, Trans. Amer. Math. Soc. {\bf 316}(1) (1989), 251--279.

\bibitem{ferreira2015} V. O. Ferreira and J. Z. Gon\c{c}alves, {\it Free symmetric and unitary pairs in division rings infinite-dimensional over their centers}, Israel J. Math. {\bf 210} (2015), 297--321.

\bibitem{goodaire2013} E. G. Goodaire and C. P. Milies, {\it Involutions and anticommutativity in group rings}, Canad. Math. Bull. {\bf 56}(2) (2013), 344--353.

\bibitem{herstein1967} I. N. Herstein, {\it Special simple rings with involution}, J. Algebra {\bf 6}(3) (1967), 369--375.

\bibitem{herstein1969} I. N. Herstein, ``Topics in ring theory'', The University of Chicago Press, Chicago, Ill.-London {\bf 1969} xi+132 pp.

\bibitem{herstein1971} I. N. Herstein and S. Montgomery, {\it A note on division rings with involutions}, Michigan Math. J. {\bf 18} (1971), 75--79.

\bibitem{herstein1974} I. N. Herstein, {\it Rings with periodic symmetric or skew elements}, J. Algebra {\bf 30} (1974), 144--154.

\bibitem{herstein1976} I. N. Herstein, ``Rings with involution", Chicago Lectures in Mathematics. University of Chicago Press, Chicago, Ill.-London, {\bf 1976}. x+247 pp.

\bibitem{herstein1983} I. N. Herstein, {\it A theorem on invariant subrings}, J. Algebra {\bf 83}(1) (1983), 26--32.

\bibitem{lanski1972}  C. Lanski and S. Montgomery, {\it Lie structure of prime rings of characteristic $2$}, Pacific J. Math. {\bf 42}(1) (1972), 117--136.

\bibitem{lee1983} P.-H. Lee and T.-K. Lee,  {\it Lie ideals of prime rings with derivations},
Bull. Inst. Math. Acad. Sinica {\bf 11}(1)  (1983), 75--80.

\bibitem{lee1986} P.-H. Lee and T.-K. Lee,  {\it  Note on nilpotent derivations},
Proc. Amer. Math. Soc. {\bf 98}(1) (1986), 31--32.

\bibitem{lee1995} T.-K. Lee,  {\it On nilpotent derivations of
semiprime rings with involution}, Chinese J. Math. {\bf 23} (1995)
155--166.

\bibitem{lee2017} T.-K. Lee, {\it Anti-automorphisms satisfying an Engel condition}, Comm. Algebra {\bf 45}(9) (2017), 4030--4036.

\bibitem{lee2018} T.-K. Lee,  {\it Commuting anti-homomorphisms}, Comm. Algebra {\bf 46}(3) (2018), 1060--1065.

\bibitem{lee2022} T.-K. Lee,  {\it On higher commutators of rings},  J. Algebra Appl. {\bf 21}(6) (2022), Paper No. 2250118, 6 pp.

\bibitem{lee2025} T.-K. Lee and J.-H. Lin, {\it Commutators and products of Lie ideals of prime rings},  arXiv:2410.00444v2 [math.RA],
https://doi.org/10.48550/arXiv.2410.00444.

\bibitem{lee2010} G. T. Lee, ``Group identities on units and symmetric units of group rings", Algebra and Applications, {\bf 12}. Springer–Verlag London, Ltd., London, 2010. xii+194 pp.

\bibitem{lee2015} G. T. Lee, S. K. Sehgal and E. Spinelli, {\it Lie identities on skew elements in group algebras}, Contemp. Math. {\bf 652} (2015), 103--121.

\bibitem{lim1979} T. P. Lim, {\it Some classes of rings with involution satisfying the standard polynomial of degree $4$}, Pacific J. Math.
{\bf 85}(1) (1979), 125--130.

\bibitem{lin1986} J.-S. Lin, {\it On derivations of prime rings with involution},
Chinese J. Math. {\bf 14}(1) (1986), 37--51.

\bibitem{lin2010} J.-S. Lin and C.-K. Liu, {\it Strong commutativity preserving maps in prime rings with involution}, Linear Algebra Appl. {\bf 432}(1) (2010), 14--23.

\bibitem{martindale1969} W. S. Martindale III, {\it Rings with involution and polynomial identities}, J. Algebra {\bf 11}(2) (1969), 186--194.

\bibitem{montgomery1971a} S. Montgomery, {\it Polynomial identity algebras with involution}, Proc. Amer. Math. Soc. {\bf 27}(1) (1971), 53--56.

\bibitem{montgomery1971b} S. Montgomery, {\it A generalization of a theorem of Jacobson}, Proc. Amer. Math. Soc. {\bf 28}(2) (1971), 366--370.

\bibitem{montgomery1973} S. Montgomery, {\it A generalization of a theorem of Jacobson II}, Pacific J. Math. {\bf 44}(1) (1973), 233--240.

\bibitem{montgomery1976} S. Montgomery, {\it A structure theorem and a positive-definiteness condition in rings with involution},
J. Algebra {\bf 43}(1) (1976), 181--192.

\bibitem{mosic2009} D. Mosi\'{c} and D. S. Djordjevi\'{c}, {\it Moore–Penrose-invertible normal and Hermitian elements in rings}, Linear Algebra Appl. {\bf 431}(5--7) (2009), 732--745.

\bibitem{osborn1967} J. M. Osborn, {\it Jordan algebras of capacity two}, Proc. Natl. Acad. Sci. U.S.A. {\bf 57}(3) (1967), 582--588.

\bibitem{posner1957} E. C. Posner, {\it Derivations in prime rings}, Proc. Amer. Math. Soc. {\bf 8}(6) (1957), 1093--1100.

\bibitem{rowen1973} L. Rowen, {\it Some results on the center of a ring with polynomial identity}, Bull. Amer. Math. Soc. {\bf 79}(1) (1973), 219--223.

\bibitem{shalev1992} A. Shalev, {\it The derived length of Lie soluble group rings I}, J. Pure Appl. Algebra {\bf 78}(3) (1992), 291--300.

\bibitem{siciliano2011} S. Siciliano, {\it On the Lie algebra of skew-symmetric elements of an enveloping algebra}, J. Pure Appl. Algebra {\bf 215}(1) (2011), 72--76.

\bibitem{thu2022} V. H. M. Thu, {\it A note on symmetric elements of division rings with involution}, Acta Math. Vietnam. {\bf 47} (2022), 495--502.

\end{thebibliography}
\end{document}